\documentclass{birkjour}
 \newtheorem{theorem}{Theorem}[section]
 
 \newtheorem{lemma}[theorem]{Lemma}
 
 \theoremstyle{definition}
 
 \theoremstyle{remark}

 \numberwithin{equation}{section}

\usepackage{hyperref}

\usepackage{amsmath,amsthm}
\usepackage{amssymb}

\usepackage{enumitem}

\usepackage{graphicx}

\usepackage{caption}
\usepackage{subcaption}

\usepackage[T1]{fontenc}
\usepackage{amsfonts}
\usepackage{float}
\usepackage{cancel}
\begin{document}

\title[Caustics of a Paraboloid and Apollonius Problem]
 {Caustics of a Paraboloid\\ and Apollonius Problem}

%----------Author 1
\author{Yagub N. Aliyev}

\address{%
School of IT and Engineering\\ 
ADA University\\
Ahmadbey Aghaoglu str. 61 \\
Baku 1008, Azerbaijan}

\email{yaliyev@ada.edu.az}

\thanks{This work was supported by ADA University Faculty Research and Development Fund.}



\date{January 1, 2004}
%\dedicatory{}
%%% ----------------------------------------------------------------------

\begin{abstract}
We study caustics of an elliptical paraboloid and the history of their various representations from 3D models in XIX century to the recent computer graphics. In the paper two ways of generating the surface, one with cartesian coordinates using formula for principal curvatures, and the other one with parabolic coordinates using Seidel's formula were demonstrated. By finding the intersection curves of these caustics with the paraboloid we extend the solution of F. Caspari for classical Apollonius problem about the number of concurrent normals to the points of the paraboloid itself. A complete classification of all possible cases of intersections of these caustics with their paraboloid is given.
\end{abstract}

%%% ----------------------------------------------------------------------
\maketitle
%%% ----------------------------------------------------------------------
%\tableofcontents
\section{Introduction} 
Apollonius in fifth book of his \textit{Conics} \cite{apol} asks this question about normals: how many concurrent normals of an ellipse can one draw? His solution, which is very intricate, revealed that there are points that play the role of a boundary, where the number of normals jump from 2 to 4. We now call the curve consisted of points with such properties as \textit{caustics}. The problem can be naturally extened for parabolas and it happens to be easier than the one for ellipses but still not trivial. The problem about the number of normals of an ellipse also appear in V. Arnold's collection of problems "Mathematical Trivium" which he considered as mathematical minimum for a physics student (see e.g. \cite{kheshin}, p. 49, 58-59). This problem  was also offered at the First All-Union Mathematical Olympiad for USSR college students in 1974 and was solved completely by only one student (see \cite{arnold}, p. 285-286). See also \cite{bains} and \cite{mcgif} for the study of this problem in connection with a cubic equation. Caustics of an ellipse and a hyperbola were studied in connection with optics in 
\cite{cayley1}. It would be waste of space here to explore a more detailed history of this problem and its generalization for an ellipsoid, as it was topic of the previous article of the author \cite{aliyev}. In more general setting the problem of the number of normals of surfaces was studied in \cite{clebsch}. From recent works in this direction one can mention \cite{hann}, \cite{dom} and their references. The number of normals has a direct interpretation as number of static balance points of the given body if the center of mass is located at $A$. This fundamental property is described, for example, in the classic book by Poston and Stewart \cite{post} (see p. 5-6). In this book we read that the connection with caustics was one of starting-points of Thom in developing his
catastrophe theory.

While studying the problem on the number of normals to a surface it soon becomes clear that the most natural way to do this task is to find the centers of principal curvatures of the surface (\cite{schroder}, p. 9; \cite{caspari}, p. 8). Although in 2 dimensional problem not using the centers of curvature can be compensated by analytic and algebraic computations, for 3 dimensions not using these curvature centers makes the problem considerably difficult to study. These centers of principal curvature were first introduced by G. Monge
\cite{monge} for completely different purposes. The particular case of the surface of centers of an ellipsoid was studied by A. Cayley \cite{cayley}. The case of hyperboloid (see \cite{dyck2} for necessary formulae) is very similar to that of the ellipsoid and with some work can be derived from \cite{aliyev}. The literature about the case of ellipsoids is extensive and references of \cite{aliyev} lists both historical sources and the recent ones.

In the current work we focus on Caustics of elliptical paraboloids and their uses in determination of the number of concurrent normals of the paraboloid. The caustics of Paraboloids together with the problem of normals were topics of doctoral dissertations by F. Caspari (1875) \cite{caspari} (see also \cite{caspari1}) and H. Schröder (1913) \cite{schroder}. The works by F. Caspari (1875) \cite{caspari} (see also \cite{caspari1}) present detailed investigation of the problem of normals in general. The cases which are studied in the current paper required more delicate consideration of the possible situations. To be more precise, the cases when the point of concurrency of normals is on the paraboloid itself are studied here with more details. This case require determination of the intersection curves of the paraboloid and its caustics, which seems have escaped the attention of the previous studies. Together with the nodal curve and the intersections of the caustics with the coordinate planes, these intersection curves provide more information about the mutual position of the paraboloid and its caustics. This idea of using the intersection curves was first explored for ellipsoids in \cite{aliyev}. The current paper is my attempt to extend the idea to the case of paraboloids. As the presentation below show, the case of paraboloids involves less number of cases than the number for ellipsoids. The solution for ellipsoids involve (to use a couple rough metrics) 12 cases for the position of caustics with respect to each other and the ellipsoid, and 8 cases for their intersections, where border cases were not counted. The solution of a similar problem for paraboloids require only 4 cases including the border one. The comparison (and the finer details of the following text) suggests that the choice of paraboloid is indeed simpler. But there is no direct analogy and it is not at all obvious how the results for the ellipsoids can give the following results for the paraboloids. In any case the 4 cases $a<b\le2a$, $2a<b<3a$, $b=3a$, and $3a<b$ highlighted in the current paper did not show up in the studies by \cite{caspari}, \cite{caspari1}, and \cite{schroder}.

Caustics of paraboloid play important role in optics, astronomy 
(see \cite{seidel}, \cite{seidel1}, \cite{seidel2}, \cite{schmidt} and its references) and mathematical theory of shipbuilding (pages 206-209 of \cite{post}). As this surface does not appear frequently in real life and nature, there were many attempts to recreate it as a drawing or as a physical 3D model. There are several models of these surfaces in various museums and private collections. One such model from M. Schilling catalog and made by L. Schleiermacher under supervision of Prof. Dr. L. Brill in TU München is in The Collection of Geometric Models
of V.N. Karazin Kharkiv National University \cite{space}. Similar models made by L. Schleiermacher are in Göttinger Sammlung mathematischer Modelle und Instrumente, Georg-August-Universität Göttingen 
\cite{goet}. Two such models from 1877 are in Museum of Kazan Federal University \cite{kazan}, which were made again by L. Schleiermacher and published as part of the first series issued by Brill \cite{brill1}. In p. 46-47 of \cite{brill1} we also find the description of a model of elliptical paraboloid $\frac{y^2}{12}+\frac{z^2}{10}-2x=0$ without a picture (see also \cite{schil}, p. 128-129). One more model from 1892 with the name of L. Brill on it is in The National Museum of American History \cite{brill}. A photo of another such model from Hallenser Collection appear in p. 264 of the dissertation of H. Junker from 2023 \cite{junker}. Some drawings of such models appear also at the end of the doctoral dissertation by H. Schröder (1913) \cite{schroder} (see also \cite{schroder1}). The caustics of a paraboloid, although not named as such, appear also on page 209 of \cite{post}.

With the dawn of computer graphics and 3D printers, these gypsum models became redundant as one can instantly generate arbitrarily many different caustic surfaces just by changing initial settings in a mathematical software and then show them on the screen of a computer, print on a paper, and even print models with a 3D printer. Seidel's formula that we use in the current paper to draw in GeoGebra were also used in \cite{kooij} to plot these caustics using Maple. Focal surface of an elliptical paraboloid was also featured in two YouTube videos by Ambjörn Naeve, who also included them in a playlist with other similar videos for hyperbolic paraboloid \cite{naev} (see  \cite{naev1} for more information). Focal surfaces of a hyperbolic paraboloid appear as an example in Wikipedia article about focal surfaces in general \cite{wiki}.

\section{Parabola}
Let us consider parabola $$y=\frac{ax^2}{2}.\eqno(1)$$Let us take a point $A(l,m)$ on its plane and find point $B(x,y)$ on the parabola for which the normal of the parabola at point $B$ passes through $A$. This happens when $$\frac{y-m}{x-l}=-\frac{1}{y'}=-\frac{1}{ax}.$$
This qives us equation of hyperbola $y=m-\frac{x-l}{ax}$, whose intersections with the parabola are points $B_i$ with normals  passing through $A$ (see Figure \ref{fig1}). The $x$-coordinates of these intersections are the solutions of cubic equation
$$
a^2x^3-2(am-1)x-2l=0.
$$
It is known that the discriminant of a depressed cubic polynomial $x^3+px+q$ is $-4p^3-27q^2$. So, the roots of our cubic equation are repeated when $$(am-1)^3=\frac{27}{8}a^2l^2.$$ This gives us semicubical  parabola $$(ay-1)^3=\frac{27}{8}a^2x^2,\eqno(2)$$also called Neile's parabola, which separates the regions of the plane where the the number of normals jumps from 1 to 3 or vice versa. Note that the normals of parabola (1) are the tangents of semicubical parabola (2) (cf. \cite{post}, p. 87). Below, on, and above the semicubical parabola, which serves also as the centers of curvature for (1), the number of normals is 1, 2, and 3, respectively. Exception is point $\left(0,\frac{1}{a}\right)$, which is focus point of parabola (1) and cusp point of semicubical parabola (2). This point is on semicubical parabola (2) but the number of normals at this point is 1. Also, on parabola (1) one of the points $B_i$ coincide with $A$. On parabola (1) the number of normals jumps from 1 to 3 or vice versa at points $C_i\left(\pm\frac{2\sqrt{2}}{a},\frac{4}{a}\right)$, which are intersection points of (1) and (2). This completely solves classical problem of Apollonius for the number of concurrent normals for parabola.
\begin{figure}
\centering

  \includegraphics[width=0.7\linewidth]{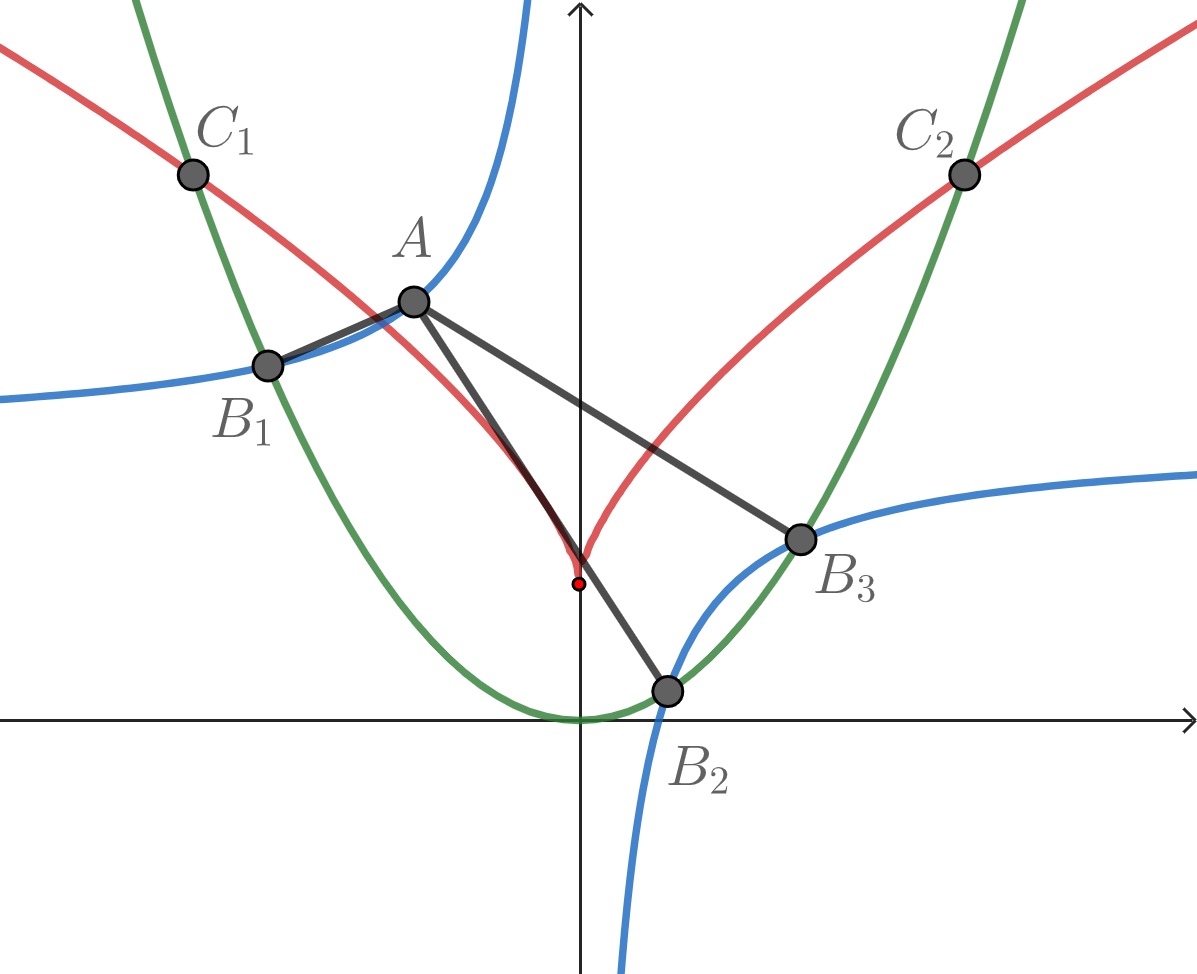}
  \captionof{figure}{Parabola $y=\frac{ax^2}{2}$ (green), hyperbola $y=m-\frac{x-l}{ax}$ (blue), semicubical parabola $(ay-1)^3=\frac{27}{8}a^2x^2$ (red), and the normal lines (black): \url{https://www.geogebra.org/calculator/qpvn9gzj}}
  \label{fig1}

\end{figure}

\section{Normals of Paraboloid}
Let us now consider paraboloid $$z=\frac{ax^2+by^2}{2}.\eqno(3)$$We will consider only the case of elliptical paraboloid ($ab>0$) as the case of the hyperbolic paraboloid ($ab<0$) can be done analogously. The cases of the paraboloid of revolution ($a=b$) and parabolic cylinder ($ab=0$) are trivial as they can be easily derived from the case of parabola in the previous section. Without loss of generality we can assume that $0<a< b$. In curvilinear coordinates $u,v$, which are similar to parabolic coordinates, this paraboloid can be parametrized as (see \cite{caspari}, p. 1; \cite{caspari1}, p. 143; \cite{schroder}, p. 7)
$$
(x(u,v))^2=\frac{b(au-1)(av+1)}{a^2(b-a)},\ (y(u,v))^2=-\frac{a(bu-1)(bv+1)}{b^2(b-a)},
$$
$$
z(u,v)=\frac{ab(u-v)-a-b}{2ab}. \eqno(4)
$$
Let us take a point $A(l,m,n)$ in space and find points $B_i(x,y,z)$ on the surface of paraboloid (3) for which the normal of the paraboloid at point $B$ passes through $A$. This happens when $$\frac{x-l}{-ax}=\frac{y-m}{-by}=z-n=t,$$
where we used the fact that the inner normal vector of paraboloid (3) at point $B(x,y,z)$ is $\vec{N}=\left({-ax},{-by},1\right)$. 
These points can be defined as intersection of parametric curve $\left(\frac{l}{1+at},\frac{m}{1+bt},n+t\right)$, where $-\infty<t<+\infty$, with paraboloid (3) (see Figure \ref{fig2}). The asymptotes of this parametric curve are $z$-axis and lines
$$
\left(t,\frac{ma}{a-b},-\frac{1}{a}+n\right), \left(\frac{lb}{b-a},-\frac{1}{b}+n,t\right)\ (-\infty<t<+\infty),
$$
which become part of the parametric curve if $l=0$ and $m=0$, respectively. Note that if $l=0$ and $m=0$, then the parametric curve itself becomes hyperbolas $\left(0,\frac{m}{1+bt},n+t\right)$ and $\left(\frac{l}{1+at},0,n+t\right)$, respectively. The values of $t$ corresponding to the intersection points are the solutions of equation (see 
\cite{smith}, p. 81)
$$
2 (t+n)=\frac{l^{2}}{a \left(t+\frac{1}{b}\right)^{2}}+\frac{m^{2}}{b \left(t+\frac{1}{a}\right)^{2}}.
$$
\begin{figure}
\centering
  \includegraphics[width=0.7\linewidth]{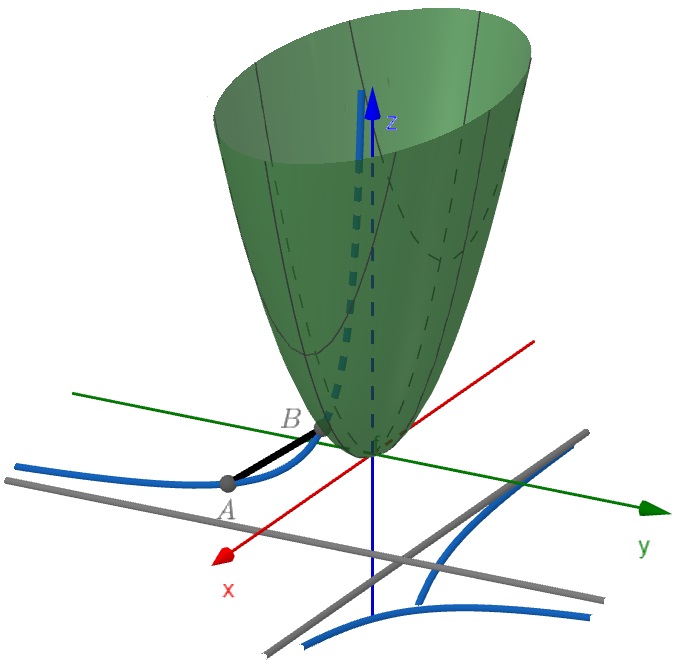}
  \captionof{figure}{Elliptical paraboloid $z=\frac{ax^2+by^2}{2}$ (green), parametric curve $\left(\frac{l}{1+at},\frac{m}{1+bt},n+t\right)$ (blue), its asymptotes (grey), and the normal line (black): \url{https://www.geogebra.org/3d/rbp3mxnc}}
  \label{fig2}
\end{figure}
If $lm\ne 0$, then there is always one such $t_i$ in interval $\left(-\frac{1}{b},+\infty\right)$. In each of intervals $\left(-\infty, -\frac{1}{a}\right)$ and $\left(-\frac{1}{a},-\frac{1}{b}\right)$, there are 0, 1, or 2 roots $t_i$ depending on the choice of $A(l,m,n)$. Note that the last equality can be written as a quintic equation with real cofficients, and therefore it has at least 1 and at most 5 real roots $t_i$, and 
consequently the number of points $B_i$ can be 1,2,3,4, or 5 depending on the choice of point $A$, and each case is realizable. The coordinates of points $B_i$ are $\left(\frac{l}{1+at_i},\frac{m}{1+bt_i},n+t_i\right)$ ($i=1,2,\ldots,5$).

\section{Caustics of paraboloid}
We will now describe points $A$ of the space for which the number of normals jumps from 1 to 3, or from 3 to 5. It is well known that points $A$, for which two normals of a surface coincide, are the centers of principal curvatures of the surface (\cite{schroder}, p. 9; \cite{caspari}, p. 8). The surface formed by these points is known by different names such as evolute, focal surface or centro-surface (surface of centers). Using the formula for principal curvatures in terms of Gaussian and mean curvatures (\cite{pat}, p. 64; \cite{mish}, p. 179, note that the definition of the mean curvature in \cite{mish} is different cf. \cite{nov}, p. 107) one can write formula for the coordinates of these centers of principal curvature of the paraboloid as (\cite{blas}, p. 88, Formula 20)
$$
(X,Y,Z)=(x,y,z)+R_{1,2}\cdot\frac{\vec{N}}{|\vec{N}|}, \eqno(5)
$$
where $R_{1,2}=\frac{1}{H\mp\sqrt{H^2-K}}$ are principal curvatures, and $H=\frac{a+b+a^2bx^2+ab^2y^2}{2\left(1+a^2x^2+b^2y^2\right)^{\frac{3}{2}}}$ and $K=\frac{ab}{(1+a^2x^2+b^2y^2)^{2}}$ are mean and Gaussian curvatures, respectively. In curvilinear coordinates $u_1,v_1$ and $u_2,v_2$, these two caustics (see Figure \ref{fig2}) can be expressed as (see \cite{caspari}, p. 2; \cite{seidel}, p. 698; \cite{schlei}, p. 5; \cite{caspari1}, p. 144; \cite{schroder}, p. 10)
$$
(x_1(u_1,v_1))^2=\frac{b(au_1-1)(av_1+1)^3}{a^2(b-a)},\ (y_1(u_1,v_1))^2=-\frac{a(bu_1-1)(bv_1+1)^3}{b^2(b-a)},
$$
$$
z_1(u_1,v_1)=\frac{ab(u_1-3v_1)-a-b}{2ab}, \eqno(6)
$$
$$
(x_2(u_2,v_2))^2=\frac{b(au_2-1)^3(av_2+1)}{a^2(b-a)},\ (y_2(u_2,v_2))^2=-\frac{a(bu_2-1)^3(bv_2+1)}{b^2(b-a)},
$$
$$
z_2(u_2,v_2)=\frac{ab(3u_2-v_2)-a-b}{2ab}. \eqno(7)
$$
Note that (7) coincides with (6) if $u_2=-v_1$ and $v_2=-u_1$. This surface is of ninth degree \cite{caspari1}, p. 148.

If point $A$ is below the two caustics, then the number of normals of the paraboloid concurrent at $A$ is 1. If point $A$ is above the two caustics, then the number is 5. Finally, if point $A$ is between the two caustics, then the number of normals is 3. If point $A$ is on the surface of the caustics then the number of normals is 2 or 4, depending on which regions the surface is a border of, between regions with 1 and 3, or between regions with 3 and 5 normals, respectively (\cite{caspari1}, p. 150). There are some exceptions to this on the coordinate planes and the intersections of the two caustics which we will discuss now.

By substituting (6) in (3) we obtain $u_1=\frac{v_1}{a v_1+b v_1+3}$. Using this in (6) and denoting $v_1=t$ we obtain parametrization of the intersection of (3) with its caustics (6) as
$$
(x_1(t))^2=\frac{b \left(b t+3\right) \left(a t+1\right)^{3}}{a^{2} \left(a-b\right) \left(\left(a+b\right) t+3\right)},\ (y_1(t))^2=-\frac{a \left(a t+3\right) \left(b t+1\right)^{3}}{ b^{2}\left(a-b\right) \left(\left(a+b\right) t+3\right)},
$$
$$
z_1(t)=-\frac{3 a b\left(a+b\right)  t^{2}+ \left(a^{2} +10 a b +b^{2} \right) t+3 (a+ b)}{2a b \left( \left(a+b\right) t+3\right) }.  \eqno(8)
$$
Similarly, by substituting (7) in (3) we obtain $v_2=-\frac{u_2}{a u_2+b u_2-3}$ and using this in (7) and denoting $u_2=s$, we obtain parametrization for intersection of (3) with (7) as
$$
(x_2(s))^2=-\frac{b \left(a s-1\right)^{3} \left(b s-3\right)}{a^{2} \left(a-b\right) \left(\left(a+b\right) s-3\right)},\ (y_2(s))^2=\frac{a \left(b s-1\right)^{3} \left(a s-3\right)}{b^{2} \left(a-b\right)\left(\left(a+b\right)s-3\right) },
$$
$$
z_2(s)=\frac{3 a b\left(a+b\right)  s^{2}-\left(a^{2} +10 a b +b^{2} \right) s+3 (a+ b)}{2a b \left( \left(a+b\right) s-3\right) },  \eqno(9)
$$
which coincides with (8) if $s$ is replaced by $-t$. By solving equation $z_1(t)=z_2(s)$ for $t$, we obtain
$$
t=\frac{8-3\left(a+ b\right) s}{\left(a+b\right) \left(\left(a+b\right) s-3\right)}.
$$
By substituting this in equation $x_1(t)=x_2(s)$ and then solving this equation for $s$ gives
$$
s_0=\frac{3 a^{2}-2 a b+3 b^{2}\pm (a-b)\sqrt{3 \left(3 a-b\right) \left(a-3 b\right)}}{2 a b \left(a+b\right)},
$$
which is real if $3a\le b$.
Substituting $s=s_0$ in (9) gives the coordinates of its self-intersection points (see Figure \ref{fig3})
$$
E_{1,2}(\pm x_2(s_0),\pm y_2(s_0),z_2(s_0)),\ E_{3,4}(\pm x_2(s_0),\mp y_2(s_0),z_2(s_0)).
$$
In particular, $
z_2(s_0)=\frac{4 \left(a-b\right)^{2}}{ab\left(a+b\right) }$. For a concise formula for the other two coordinates $x_2(s_0),\  y_2(s_0)$ it is more convenient to use a different method. We will return to this question in the next section.
\begin{figure}
\centering
  \includegraphics[width=0.7\linewidth]{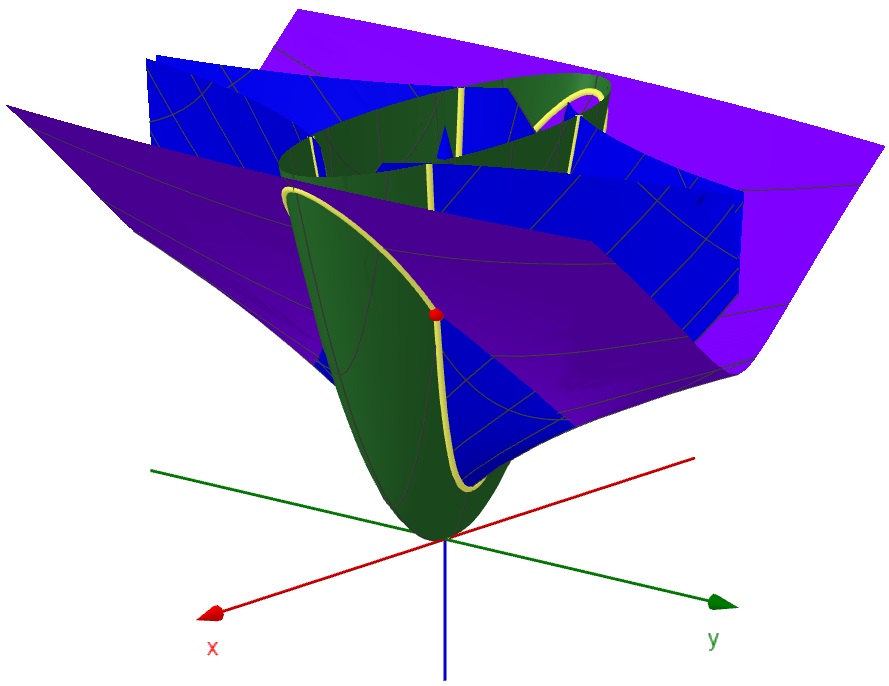}
  \captionof{figure}{Elliptical paraboloid $z=\frac{ax^2+by^2}{2}$ (green), the intersection curves (yellow) of the paraboloid with its caustics (violet and blue), the self-intersection points (red) of these curves: \url{https://www.geogebra.org/3d/p7vrgv6w}}
  \label{fig3}
\end{figure}

\section{Nodal curve of Paraboloid}

We will now find a parametrization for the intersection curve of caustics (6) and (7), which can also be interpreted as self-intersection of (6) or (7). This intersection curve is known as \textit{nodal curve} (see \cite{cayley}, p. 332 for ellipsoid). First we solve system of equations
$$
\{x_1(u_1,v_1)=x_2(u_2,v_2),\ y_1(u_1,v_1)=y_2(u_2,v_2),\}
$$
for $u_1$ and $v_2$, to obtain
$$
v_2=\frac{a\left(1+\left(\frac{bu_2-1}{bv_1+1}\right)^3\right)-b\left(1+\left(\frac{au_2-1}{av_1+1}\right)^3\right)}{ab\left(\left(\frac{au_2-1}{av_1+1}\right)^3-\left(\frac{bu_2-1}{bv_1+1}\right)^3\right)},  \eqno(10)
$$
and a similar formula for  $u_1$. Then we substitute these formula in equation $
z_1(u_1,v_1)=z_2(u_2,v_2),$ denote $v_1=t$ and solve equation $
z_1(u_1,v_1)=z_2(u_2,v_2),$ for $u_2$ to obtain
$$
u_2(t)=\frac{4 a b t+3 a+3 b+\sqrt{12 a^{2} b^{2} t^{2}+12 a^{2} b t+12 a \,b^{2} t+9 a^{2}-6 a b+9 b^{2}}}{2 a b}.   \eqno(11)
$$
Note that the discriminant of the quadratic function under the radical is $D=-288 a^{2} b^{2} \left(a-b\right)^{2}<0$. Then equality (10) gives
$$
v_2(t)=\frac{a\left(1+\left(\frac{bu_2(t)-1}{bt+1}\right)^3\right)-b\left(1+\left(\frac{au_2(t)-1}{at+1}\right)^3\right)}{ab\left(\left(\frac{au_2(t)-1}{at+1}\right)^3-\left(\frac{bu_2(t)-1}{bt+1}\right)^3\right)}.  \eqno(12)
$$
From (7) it follows that the nodal curve can be parametrized as
$$
(x(t))^2=\frac{b(au_2(t)-1)^3(av_2(t)+1)}{a^2(b-a)},\ (y(t))^2=-\frac{a(bu_2(t)-1)^3(bv_2(t)+1)}{b^2(b-a)},
$$
$$
z(t)=\frac{ab(3u_2(t)-v_2(t))-a-b}{2ab}, \eqno(13)
$$
where $u_2(t)$ and $v_2(t)$ are defined by (11) and (12), respectively. The nodal curve also passes through intersection points $E_i$ of curves (8) and (9) whenever points $E_i$ exist (see Figure \ref{fig4}). The conditions for the existence of points $E_i$ will be obtained in the next section. A different parametrization for the nodal curve of the paraboloid is given in \cite{caspari1}, p. 185 (cf. \cite{schlei}, p. 10):
$$
(x(t))^2=\frac{8a\delta^2(t+\delta)^3(t-2\delta)^2}{t^4},\ (y(t))^2=\frac{8b\delta^2(t-\delta)^3(t+2\delta)^2}{t^4},
$$
$$
z(t)=\frac{8\delta^2-t^2}{t}+\frac{a+b}{2ab},
$$
where $\delta=\frac{b-a}{2ab}$. Note that the formula for $z$-coordinate is slightly shifted in comparison to \cite{caspari1}, p. 185. It is similar to Cayley's parametrization for the nodal curve of an ellipsoid in \cite{cayley}, p. 351, which was used in \cite{aliyev} to solve similar questions for a triaxial ellipsoid. By solving equation $
z(t)=z_2(s_0)$ or, which is the same,
$$
\frac{8\delta^2-t^2}{t}+\frac{a+b}{2ab}=\frac{4 \left(a-b\right)^{2}}{ab\left(a+b\right)},
$$
we obtain $t_0=\frac{a+b}{2 a b}$. Using this in the last parametrization we obtain
$$
x_2(s_0)=x(t_0)=\frac{2\sqrt{2}(b-a)(b-3a)}{a(a+b)^2},\ y_2(s_0)=y(t_0)=\frac{2\sqrt{2}(b-a)(3b-a)}{b(a+b)^2},
$$
where $b\ge3a$. Thus all the coordinates of points $E_i$ are now calculated:
$$
E_{1,2}\left(\pm \frac{2\sqrt{2}(b-a)(b-3a)}{a(a+b)^2},\pm \frac{2\sqrt{2}(b-a)(3b-a)}{b(a+b)^2},\frac{4 \left(a-b\right)^{2}}{ab\left(a+b\right)}\right),
$$
$$
E_{3,4}\left(\pm \frac{2\sqrt{2}(b-a)(b-3a)}{a(a+b)^2},\mp \frac{2\sqrt{2}(b-a)(3b-a)}{b(a+b)^2},\frac{4 \left(a-b\right)^{2}}{ab\left(a+b\right)}\right).
$$
The other solution $t_1=-\frac{4 \left(a-b\right)^2}{a b \left(a+b\right)}$ of equation $
z(t)=z_2(s_0)$ does not give a real point.
\begin{figure}
\centering
  \includegraphics[width=0.8\linewidth]{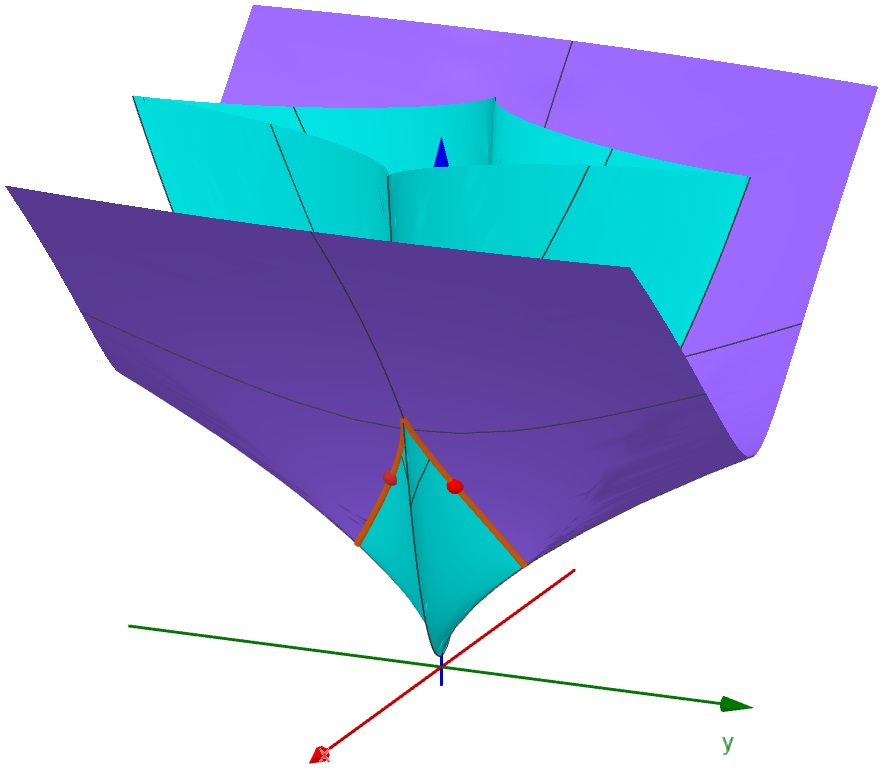}
  \captionof{figure}{Caustics (purple and blue) of paraboloid, the nodal curve (orange), and its intersection points $E_i$ (red) with the paraboloid (not shown): \url{https://www.geogebra.org/3d/d66y3gjw}}
  \label{fig4}
\end{figure}
If point $A$ is on the nodal curve, then the number of normals through point $A$ of the paraboloid is 3 \cite{caspari1}, p. 176, except each of its endpoints $H_i$ and $K_i$, where the number of normals is 2. This will be discussed in the next section.

\section{Intersections with coordinate planes}
We now focus on coordinate planes $x=0$ and $y=0$. Paraboloid (3) intersects these planes along parabolas $z=\frac{by^2}{2}$,  $x=0$, and $z=\frac{ax^2}{2}$, $y=0$. By substituting, for example, $u=\frac{1}{a}$ and $u=\frac{1}{b}$, together with notation $v=t$ in (4), these parabolas can also be written parametrically as
$$
x(t)=0,\ (y(t))^2=-\frac{bt+1}{b^2},\ z(t)=-\frac{bt+1}{2b}, \eqno(14)
$$
$$
(x(t))^2=-\frac{at+1}{a^2},\ y(t)=0,\ z(t)=-\frac{at+1}{2a}, \eqno(15)
$$
respectively. By substituting $u_1=\frac{1}{a}$ and $u_2=\frac{1}{a}$, together with notations $v_1=t$ and $v_2=t$ in (6) and (7), we obtain
$$
x(t)=0,\ (y(t))^2=-\frac{(bt+1)^3}{b^2},\ z(t)=-\frac{3bt+1}{2b}, \eqno(16)
$$
$$
x(t)=0,\ (y(t))^2=-\frac{(b-a)^2(bt+1)}{a^2b^2},\ z(t)=-\frac{abt+a-2b}{2ab}, \eqno(17)
$$
respectively. Similarly, by substituting $u_1=\frac{1}{b}$ and $u_2=\frac{1}{b}$, together with notations $v_1=t$ and $v_2=t$ in (6) and (7), we obtain
$$
(x(t))^2=-\frac{(at+1)^3}{a^2},\ y(t)=0,\ z(t)=-\frac{3at+1}{2a}, \eqno(18)
$$
$$
(x(t))^2=-\frac{(b-a)^2(at+1)}{a^2b^2},\ y(t)=0,\ z(t)=-\frac{abt+b-2a}{2ab}. \eqno(19)
$$
Note that (16) and (18) are semicubical parabolas, while (17) and (19) are just parabolas. On each plane some of these curves intersect and the resulting intersection points are shown in Table \ref{tab1}.
\tiny
\begin{table} [h!]
\begin{center}
\begin{tabular}{||c||c|c|c||} 
 \hline
 & (14) & (16) & (17) \\ [0.5ex] 
 \hline\hline
 (14) & -- & $F_1\left(0,\frac{2\sqrt{2}}{b},\frac{4}{b}\right)$ &  $G_1\left(0,\sqrt{\frac{2(a-b)^2}{ab^2(b-2a)}},\frac{(a-b)^2}{ab(b-2a)}\right)$  \\ 
 \hline
 (16) &  $F_2\left(0,-\frac{2\sqrt{2}}{b},\frac{4}{b}\right)$  & -- &  $H_1\left(0,\sqrt{\frac{(b-a)^3}{a^3b^2}},\frac{3b-a}{2ab}\right)$  \\
 \hline
 (17) &  $G_2\left(0,-\sqrt{\frac{2(a-b)^2}{ab^2(b-2a)}},\frac{(a-b)^2}{ab(b-2a)}\right)$  &  $H_2\left(0,-\sqrt{\frac{(b-a)^3}{a^3b^2}},\frac{3b-a}{2ab}\right)$  & -- \\ [1ex] 
 \hline
\end{tabular}

\end{center}

\begin{center}
\begin{tabular}{||c||c|c|c||} 
 \hline
 & (15) & (18) & (19) \\ [0.5ex] 
 \hline\hline
 (15) & -- &  $I_1\left(\frac{2\sqrt{2}}{a},0,\frac{4}{a}\right)$  & $J_1\left(\sqrt{\frac{2(a-b)^2}{a^2b(a-2b)}},0,\frac{(a-b)^2}{ab(a-2b)}\right)$ \\ 
 \hline
 (18) & $I_2\left(-\frac{2\sqrt{2}}{a},0,\frac{4}{a}\right)$ & -- & $K_1\left(\sqrt{\frac{8(b-a)^3}{a^2b^3}},0,\frac{4b-3a}{ab}\right)$ \\
 \hline
 (19) & $J_2\left(-\sqrt{\frac{2(a-b)^2}{a^2b(a-2b)}},0,\frac{(a-b)^2}{ab(a-2b)}\right)$ & $K_2\left(-\sqrt{\frac{8(b-a)^3}{a^2b^3}},0,\frac{4b-3a}{ab}\right)$ & -- \\ [1ex] 
 \hline
\end{tabular}

\end{center}
    \caption{Intersections with the coordinate planes.}
      \label{tab1}
    \end{table}
\normalsize
Note that $H_1$ and $H_2$ are the tangency points of (16) and (17), and they are above paraboloid (3) iff $3a>b$. The nodal curve connects points $H_i$ (umbilical center) and $K_i$ (outcrop) (cf. \cite{cayley}, p. 327). Therefore, points $E_i$ are real if and only if $3a\le b$. If $3a=b$, then points $E_i$, $F_i$, $G_i$, and $H_i$ coincide. From the table we can see that $G_1$ and $G_2$ are real and finite points if and only if $b>2a$. Note also that since $a<b$, $J_1$ and $J_2$ are not real points. We can also notice that the vertex of curves (16) and (19) is point $\left(0,0,\frac{1}{b}\right)$. Similarly, the vertex of curves (17) and (18) is point $\left(0,0,\frac{1}{a}\right)$.
\begin{figure}
\centering
  \includegraphics[width=0.8\linewidth]{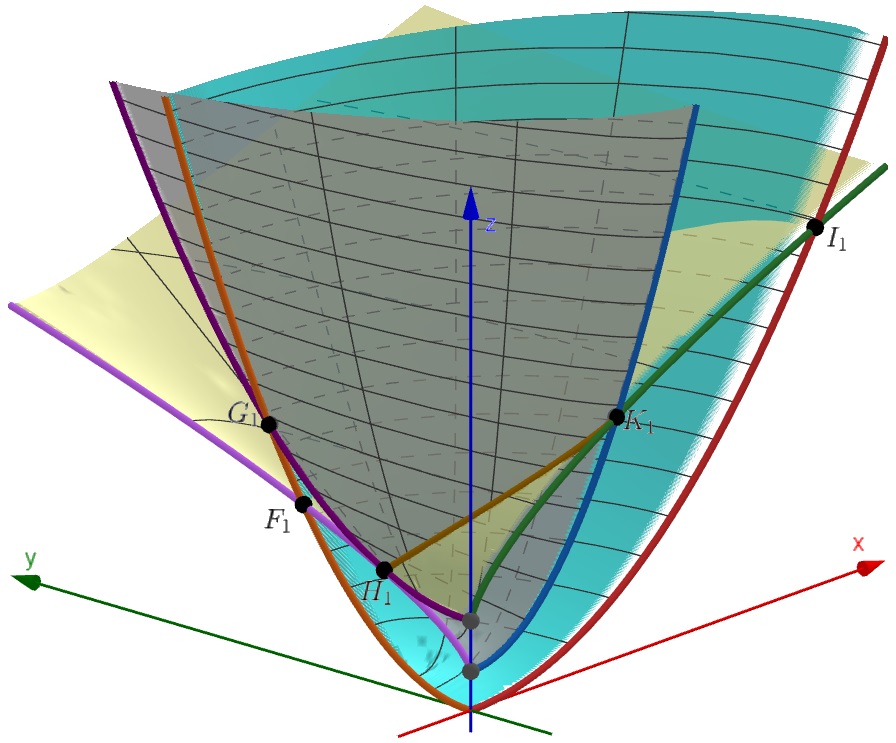}
  \captionof{figure}{Only the first quadrant: the caustics (gray and yellow) of paraboloid (cyan), curves (14) (orange), (15) (red), (16) (magenta), (17) (purple), (18) (green), (19) (blue), the nodal curve (brown), and their intersection points (black): \url{https://www.geogebra.org/3d/gt5pj7jn}}
  \label{fig5}
\end{figure}
The number of normals on curves (16)-(19) can be found using the asymptotes and the hyperbolas introduced in Section 3. By considering the cases of tangency of the asymptotes and the hyperbolas with the paraboloid in each of cases $l=0$ and $m=0$ we obtain the following result (cf. \cite{caspari1}, p. 179).
\begin{lemma} 
\begin{enumerate}
\item Along semicubical parabola (16), the number of normals is 2 above points $H_i$, 2 at points $H_i$, and 2 below points $H_i$.

\item Along parabola (17), the number of normals is 3 above points $H_i$, 2 at points $H_i$, and 3 below points $H_i$.

\item Along semicubical parabola (18), the number of normals is 2 above points $K_i$, 2 at points $K_i$, and 4 below points $K_i$.

\item Along parabola (19), the number of normals is 3 above points $K_i$, 2 at points $K_i$, and 1 below points $K_i$.
\end{enumerate}
\end{lemma}
Note also that at points $\left(0,0,\frac{1}{b}\right)$ and $\left(0,0,\frac{1}{a}\right)$, the number of normals is 1 and 3, respectively (cf. \cite{caspari1}, p. 175, footnote).
\section{Main results}
We are now ready to classify all possible cases of the positions of the caustics with respect to their paraboloid. The following theorem also solves Apollonius problem for the points of the paraboloid itself by considering all possible cases.

\begin{theorem} 
\begin{enumerate}

\item If $a<b\le2a$, then only one of the caustics intersect the paraboloid and the intersection curve divide the paraboloid into 2 regions, where the number of normals is 1 (lower) and 3 (upper region). On the intersection curve the number of normals is 2.

\item If $2a<b<3a$, then both of the caustics intersect the paraboloid and the intersection curves divide the paraboloid into 4 regions, where the number of normals is 1, 3 and 5. On the intersection curves the number of normals is 2 (lower) or 4 (upper curve), except points $G_i$, where the number of normals is 3.

\item If $b=3a$, then both of the caustics intersect the paraboloid, the intersection curves themselves intersect at two points and divide the paraboloid into 5 regions, where the number of normals is 1 (lower), 3 (middle) and 5 (upper region). On the intersection curves the number of normals is 2 (lower) or 4 (upper curve), including points $E_1=E_2$ and $E_3=E_4$, which in this case also coincide with $F_i$, $G_i$, and $H_i$ for $i=1,2$, respectively, where the number of normals is 2.

\item If $3a<b$, then both of the caustics intersect the paraboloid, the intersection curves themselves intersect at four points and divide the paraboloid into 7 regions, where the number of normals is 1 (lower), 3 (middle), and 5 (upper region). On the intersection curves the number of normals is 2 (lower) or 4 (upper parts of the curves), except points $E_i$ ($i=1,2,3,4$) and $G_i$ ($i=1,2$), where the number of normals is again 3.
\end{enumerate}
\end{theorem}
\begin{proof}
If point $A$ is below, between and above the two caustics, then the number of normals is 1, 3, and 5, respectively. On the caustics a pair of normals coincide, and therefore for the points of the caustics the number of normals is 2 or 4, except the points of the coordinate planes and the intersection (nodal) curves of these two caustics where the number of normals can be again 1 or 3. If point $A$ is on the nodal curve, then two pairs of normals coincide, and therefore the number of normals of paraboloid (3) concurrent at point $A$ is 3, except the endpoints of the nodal curve, where it is again 2. The nodal curve intersects the paraboloid at points $E_i$, and therefore the number of normals at points $E_i$ is 3 when $3a<b$ and 2 when $b=3a$. The other cases follow from the cases, which were listed in Lemma 6.1. It remains only to note that if $b\le3a$, then point $G_i$ is above point $H_i$ or coincides with point $H_i$, and if $3a<b$, then point $G_i$ is below point $H_i$.
\end{proof}
\begin{figure}
\centering
  \includegraphics[width=0.4\linewidth]{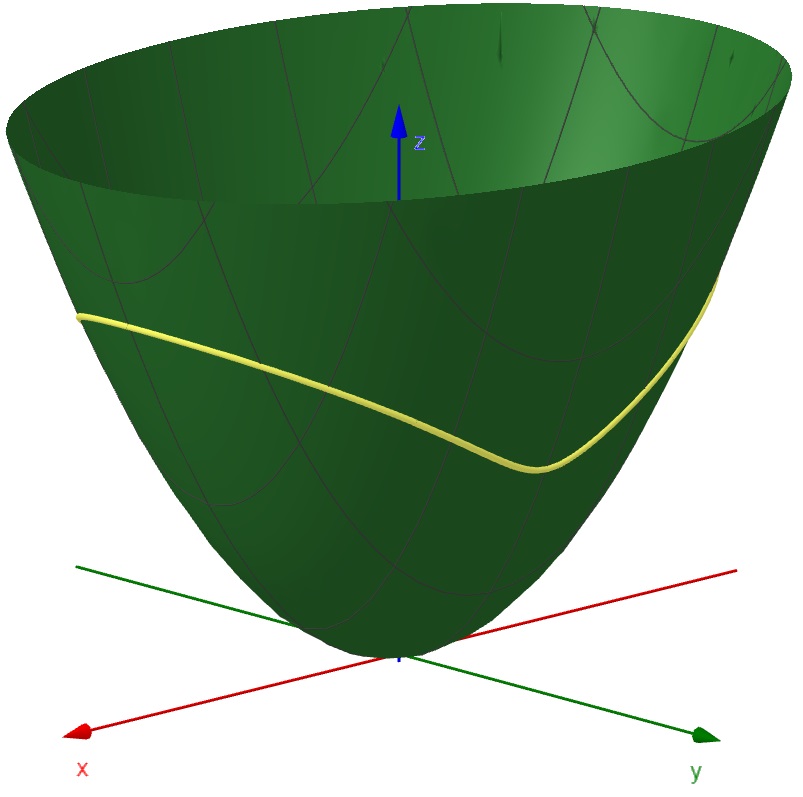}  \includegraphics[width=0.4\linewidth]{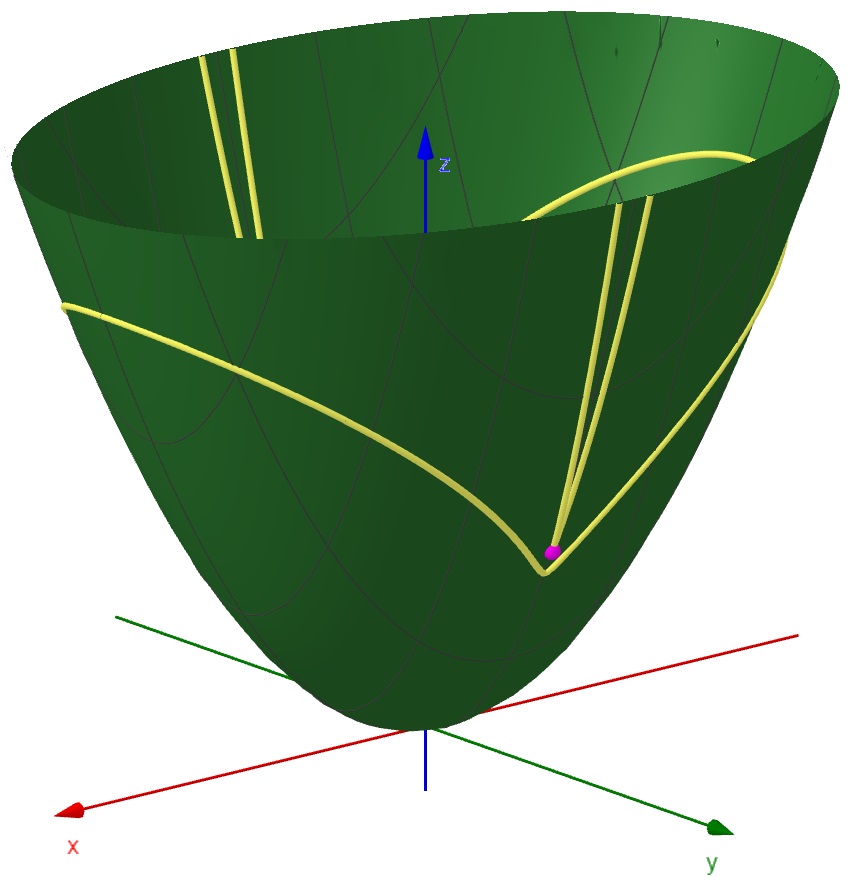}\\
  \includegraphics[width=0.4\linewidth]{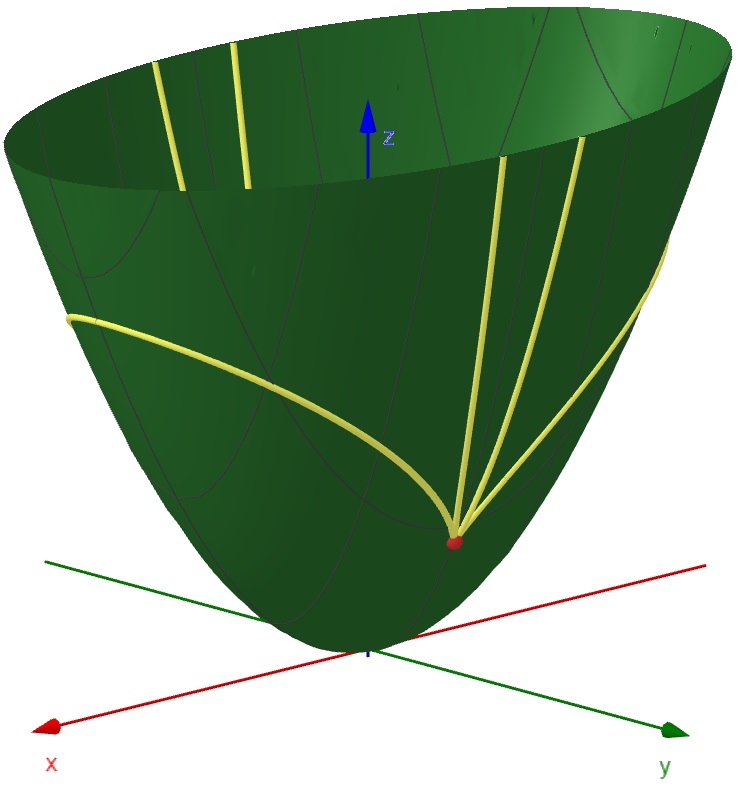}   \includegraphics[width=0.4\linewidth]{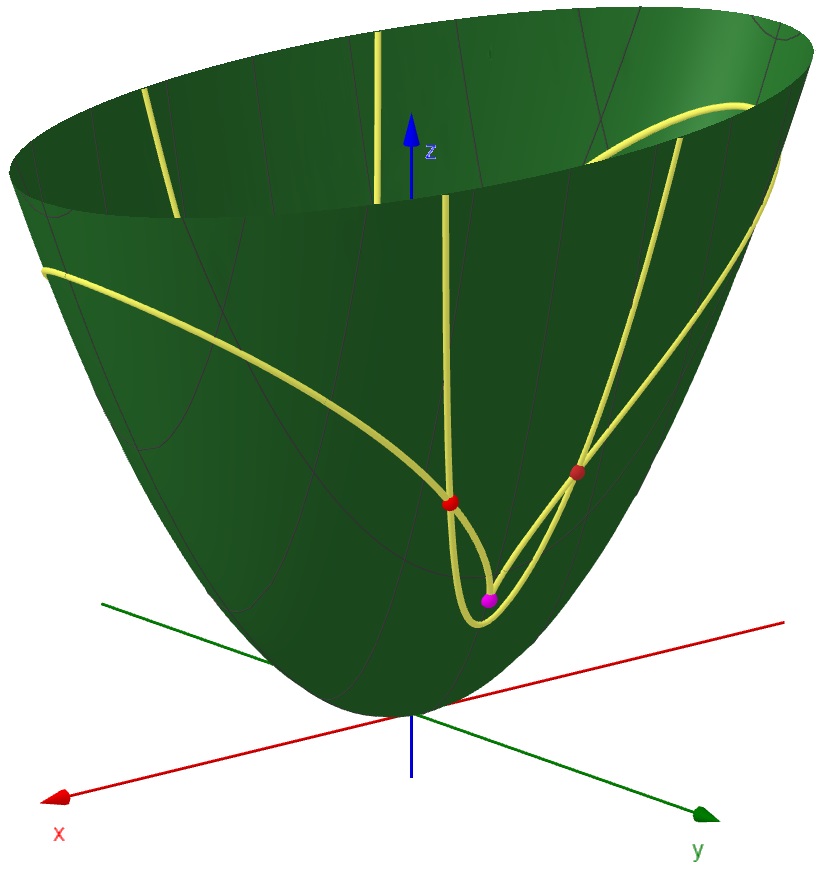}
  \captionof{figure}{$a<b\le2a$ (top left), $2a<b<3a$ (top right), $b=3a$ (bottom left), and $3a<b$ (bottom right). Paraboloid (green), the intersection curves (yellow) with its caustics (not shown), points $E_i$ (red), and points $G_i$ (magenta).}
  \label{fig6}
\end{figure}
Note that in our classification we ignored non-real or infinite points (cf. \cite{caspari1}, p. 174). For example, when $b=2a$, there is an infinite intersection point $G_i$ of one of the caustics with the paraboloid. But the resulting picture is not different from Figure \ref{fig6} (top left). So, for the classification purposes, one can safely join it with the case $a<b<2a$.
\section{Conclusion}
In the paper the history of caustic surfaces of a paraboloid and their uses in the solution of F. Caspari for the classical problem of Apollonius about the number of normals of an elliptical paraboloid passing through a given point is presented. The problem is then extended for the points of the paraboloid itself. The solution to this more specific question is completed with the help of some new curves which are obtained when the paraboloid is intersected with its caustics. New parametrizations for these curves and also for the nodal curve are given. Several different ways of drawing the paraboloid's caustics and their intersections in GeoGebra are demonstrated. Full classification of all possible cases of the intersections of the paraboloid and its caustics is given.

\section*{Acknowledgments}
%This work was supported by ADA University Faculty Research and Development Fund.

\section{Declarations}
\textbf{Ethical Approval.}
Not applicable.
 \newline \textbf{Competing interests.}
None.
  \newline \textbf{Authors' contributions.} 
Not applicable.
  \newline \textbf{Funding.}
This work was completed with the support of ADA University Faculty Research and Development Fund.
  \newline \textbf{Availability of data and materials.}
Not applicable
% ------------------------------------------------------------------------

% ------------------------------------------------------------------------

\begin{thebibliography}{1}

\footnotesize
\baselineskip=17pt

\bibitem{aliyev} Aliyev, Y.N., Apollonius Problem and Caustics of an Ellipsoid, International Electronic Journal of Geometry, in Press, 2024. Available as preprint in Arxiv:
\url{https://doi.org/10.48550/arXiv.2305.06065}

\bibitem{apol} Apollonius of Perga, \textit{Conics Books V to VII, The Arabic Translation of the Lost Greek Original in the Version of the Banū Mūsā}, Gerald J. Toomer (ed.), Springer New York, NY (1990). \url{https://doi.org/10.1007/978-1-4613-8985-9}

\bibitem{arnold} V.I. Arnol'd, A.A. Kirillov, V.M. Tikhomirov, M.A. Shubin, On First All-Union Mathematical Competition for Students, Uspekhi Mat. Nauk, 30:4(184) (1975), 281-288. (in Russian) \url{https://www.mathnet.ru/rus/rm4283}

\bibitem{bains} M.S. Bains and J.B. Thoo,
The Normals to a Parabola and the Real Roots of a Cubic, The College Mathematics Journal , Sep., 2007, Vol. 38, No. 4 (Sep., 2007), pp. 272-277. \url{https://www.jstor.org/stable/27646502}

\bibitem{blas} W. Blaschke, Einführung in die Differentialgeometrie, Springer, Berlin, 1950.

\bibitem{brill} L. Brill, Central Surface of a Paraboloid, Geometric Model,  No. 149. Ser. 1, No. 2a, The National Museum of American History (1892). \url{https://americanhistory.si.edu/collections/nmah_693993}

\bibitem{brill1} L. Brill, Catalog mathematischer Modelle
für den höheren mathematischen Unterricht,
Dritte Auflage. Darmstadt. 1885. \url{https://opendigi.ub.uni-tuebingen.de/opendigi/BRILL#p=264}

\bibitem{caspari} F. Caspari, Die Krümmungsmittelpunktsfläche des elliptischen Paraboloids, Dissert., Reimer, Berlin
(1875). \url{http://resolver.sub.uni-goettingen.de/purl?PPN310966825}

\bibitem{caspari1} F. Caspari, Die Krümmungsmittelpunktsflächen des elliptischen Paraboloids. Journal für die reine und angewandte Mathematik, Band 81, 1876, 143-192. \url{https://resolver.sub.uni-goettingen.de/purl?PPN243919689_0081}

\bibitem{clebsch} A. Clebsch, Ueber das Problem der Normalen bei Curven und Oberflächen der zweiten Ordnung, \textit{Journal für die reine und angewandte Mathematik}, \textbf{62}, 64-109 (1863). \url{http://eudml.org/doc/147884}

\bibitem{cayley} A. Cayley, On the centro-surface of an ellipsoid, \textit{Transactions of the Cambridge Philosophical Society}, \textbf{12}(1), 319-365 (1873). Also included in \textit{The collected mathematical papers of Arthur Cayley}, Vol. VIII, Cambridge University Press, Cambridge, 316-365 (1895). \url{http://name.umdl.umich.edu/ABS3153.0008.001}

\bibitem{cayley1} A. Cayley, XIII. A memoir upon caustics, Phil. Trans. R. Soc.147, (1857) 273–312.
\url{http://doi.org/10.1098/rstl.1857.0014}

\bibitem{dom} G. Domokos, Z. Lángi, T. Szabó, A topological classification of convex bodies, \textit{Geom. Dedicata} \textbf{182}, 95–116 (2016). \url{https://doi.org/10.1007/s10711-015-0130-4}

\bibitem{space} K.D. Drach, V. Komlev, Space evolute of an elliptic paraboloid and a one-sheeted hyperboloid of M. Schilling catalogue, Category: Space caustics of quadrics, Geometric Models Collection
of V.N. Karazin Kharkiv National University.

\url{http://touch-geometry.karazin.ua/m/space-evolute-of-an-elliptic-paraboloid}

\bibitem{dyck2} W. Dyck, Die Centralfläche des einschaligen Hyperboloids, \textit{Abhandlungen und Erläuterungen zu den mathematischen Modellen der Serien I-XII des Modellverlags}, unter Leitung von L. Brill, 13-18, Darmstadt (1877-1885). \url{https://opendigi.ub.uni-tuebingen.de/opendigi/BRILL#tab=struct&p=21}

\bibitem{goet} Curvature centre point surface of the elliptic paraboloid. Krümmungsmittelpunktsfläche Modellen: 345, Gypsum; Göttinger Sammlung mathematischer Modelle und Instrumente,  Georg-August-Universität Göttingen.

\url{https://sammlungen.uni-goettingen.de/objekt/record_DE-MUS-069123_345/1/-/}

\bibitem{hann} K. Hann, What's the Bound on the Average Number of Normals? The American Mathematical Monthly, Vol. 103, No. 10 (Dec., 1996), pp. 897-900. \url{https://www.jstor.org/stable/2974616}

\bibitem{junker} H. Junker, \textit{Anschauungsmodelle in der mathematischen Forschung deutscher Gelehrter 1860–1877}, Dissert., Martin-Luther-Universität
Halle-Wittenberg (2023). \url{https://opendata.uni-halle.de/bitstream/1981185920/110975/1/Dissertation_MLU_2023_JunkerHannes.pdf}

\bibitem{kazan} 3D Model, Geometric model Central Surface of a Paraboloid, Museum of Kazan Federal University.

\url{https://sketchfab.com/3d-models/geometric-model-central-surface-of-a-paraboloid-370bcc66f0924ab1b81abf76b5382fe4}

\url{https://sketchfab.com/3d-models/geometric-model-central-surface-of-a-paraboloid-5822c6b7c97c42659e455d433078b97f}

\url{https://zenodo.org/records/10227546}

\bibitem{kheshin} B.A. Khesin, S.L. Tabachnikov (Ed.). Arnold: Swimming Against the Tide, AMS, 2014.

\bibitem{kooij} J.F. Kooij, Caustics:
The von Seidel equations, Master's Thesis under the supervision of Prof. dr. J. Top
and Dr. A. E. Sterk, University of Groningen, October 2016. \url{https://fse.studenttheses.ub.rug.nl/14823/1/Masterthesis_Josselin.pdf}

\bibitem{mcgif} J. McGiffert, Normals to the Parabola, Mathematics News Letter, Vol. 7, No. 6 (Mar., 1933), pp. 12-17. \url{https://www.jstor.org/stable/3027768}

\bibitem{mish} A. Mishchenko, A. Fomenko, A Course Of Differential Geometry And Topology, Mir Publishers, Moscow, 1988.

\bibitem{monge} G. Monge, \textit{Application de l’Analyse à la Géométrie} (5e édition), Paris (1850). \url{https://gallica.bnf.fr/ark:/12148/bpt6k96431405/f344.item}

\bibitem{naev} A. Naeve, Focal surface of elliptic paraboloid (total and patch), YouTube videos, Feb 27, 2015. \url{https://youtu.be/McxKUQSPkcQ?si=lFCfKeKHYnP6taxC} and \url{https://youtu.be/pCqIV9FMpds?si=8-XEeRX9LI372IzR}. Playlist: \url{https://youtube.com/playlist?list=PL2B46F24ED8AB1FD2&si=V1VKmuPEdbZTlmSq}

\bibitem{naev1} A. Naeve, Focal Surfaces, the website of the Knowledge Management Research Group. \url{https://kmr.dialectica.se/wp/research/math-rehab/learning-object-repository/geometry-2/metric-geometry/euclidean-geometry/geometric-optics/focal-surfaces/}

\bibitem{nov} S.P. Novikov, A.T. Fomenko, Basic Elements of Differential Geometry and Topology, Mathematics and its Applications series, volume 60, Springer, Dordrecht, 1990.

\bibitem{pat} N.M. Patrikalakis, T. Maekawa, Shape Interrogation for Computer
Aided Design and Manufacturing, Springer, Berlin, 2002.

\bibitem{post} T. Poston, I. Stewart: Catastrophe Theory and its Applications. Series: Surveys and Reference Works in Mathematics, 2, Pitman, 1979.

\bibitem{schil} Martin Schilling, Catalog
mathematischer Modelle
für den hoheren mathematischen Unterricht, Siebente Auflage, Verlag von
Martin Schilling, Leipzig, 1911. \url{https://libsysdigi.library.uiuc.edu/ilharvest/mathmodels/0006cata/0006CATA.pdf}

\bibitem{schlei} L. Schleiermacher, Die Brennfläche eines Strahlensystems, welche mit der
Fläche der Krümmungscentra des elliptischen Parahboloids in collinearer Verwandtschaft steht, \textit{Abhandlungen und Erläuterungen zu den mathematischen Modellen der Serien I-XII des Modellverlags}, unter Leitung von L. Brill, 5-11, Darmstadt (1877-1885). \url{https://opendigi.ub.uni-tuebingen.de/opendigi/BRILL#p=13}

\bibitem{schmidt} R.F. Schmidt, Analytical Caustic Surfaces, NASA, Technical Memorandum 87805, 1987. \url{https://ntrs.nasa.gov/citations/19880001678} and \url{https://core.ac.uk/download/pdf/42834678.pdf}

\bibitem{schroder} H. Schröder, \textit{Die Zentraflächen der Paraboloide und Mittelpunktsflächen zweiten Grades}, Dissert., Halle a. S., H. John (1913).
\url{http://resolver.sub.uni-goettingen.de/purl?PPN316295612}

\bibitem{schroder1} H. Schröder, Die Zentraflächen der Paraboloide und Mittelpunktsflächen zweiten Grades, Abhandlung zu den Modellen Serie XLIII, Nr. 1-7, Leipzig, Verlag von Martin Schilling, 1913.

\bibitem{seidel} P.L. von Seidel, E. Kummer, Uber die Brennfl\"{a}che eines Strahlenb\"{u}ndels, welches durch
ein System von centrirten sph\"{a}rischen Gl\"{a}sern hindurch gegangen ist. Monatsberichte der Königlichen Preussische Akademie des Wissenschaften zu Berlin,
pages 695–705, 1862. \url{https://www.biodiversitylibrary.org/item/112407#page/756/mode/1up}

\bibitem{seidel1} Seidel, Ueber die Theorie der kaustischen Flächen, welche
in Folge der Spiegelung oder Brechung von Strahlenbüscheln
an den Flächen eines optischen Apparats erzeugt werden Die Fortschritte der Physik : dargest. von d. Physikalischen Gesellschaft zu Berlin. 13. 1857 (1859) 212-214. \url{https://www.bavarikon.de/object/bav:BSB-MDZ-00000BSB10707401?p=1&cq=Die+Fortschritte+der+Physik&lang=de}

\bibitem{seidel2} L. Seidel, Ueber die Entwicklung der Glieder 3ter Ordnung, welche den Weg eines ausserhalb der Ebene der Axe gelegenen Lichtstrahles durch ein System brechender Medien bestimmen, Zur Dioptrik. 1856, Astronomische Nachrichten Nr. 1027-1029. pages 289-332. \url{https://scholar.archive.org/search?q=L.+Seidel+Dioptrik}

\bibitem{smith} Ch. Smith, An Elementary Treatise on Solid Geometry, Macmillan, London, 1907.

\bibitem{wiki} Focal surface, Wikipedia, the free encyclopedia. \url{https://en.wikipedia.org/wiki/Focal_surface}

\end{thebibliography}
\end{document}